\documentclass[a4paper,11pt,twoside]{amsart}

\usepackage[left=2.7cm,right=2.7cm,top=3.5cm,bottom=3cm]{geometry}

\usepackage{amsmath,amssymb,amscd,amsfonts}
\usepackage{mathrsfs}
\usepackage{latexsym}
\usepackage{graphicx}
\usepackage[all]{xy}
\usepackage{eucal}
\usepackage{verbatim}

\newcommand{\Q}{\mathbb{Q}}

\newcommand{\Z}{\mathbb{Z}}

\newcommand{\F}{\mathbb{F}} 
\newcommand{\M}{\mathfrak{M}} 
\renewcommand{\k}{\kappa} 
\newcommand{\G}{\mathcal{G}} 
\newcommand{\V}{\mathcal{V}}
\newcommand{\X}{\mathcal{X}}
\newcommand{\Gal}{{\rm Gal}}
\newcommand{\sri}{\twoheadrightarrow}
\newcommand{\iri}{\hookrightarrow}
\renewcommand{\l}{\ell}

\renewcommand{\L}{\Lambda}
\newcommand{\wh}{\widehat}
\newcommand{\ov}{\overline}
\newcommand{\corank}{{\rm corank}}
\newcommand{\rank}{{\rm rank}}
\newcommand{\dl}[1]{\lim_{\buildrel \longrightarrow\over{#1}}}
\newcommand{\il}[1]{\lim_{\buildrel \longleftarrow\over{#1}}}

\newtheorem{defin}{Definition}[section]
\newtheorem{prop}[defin]{Proposition}
\newtheorem{lem}[defin]{Lemma}
\newtheorem{rem}[defin]{Remark}
\newtheorem{thm}[defin]{Theorem}

\newtheorem{cor}[defin]{Corollary}
\newtheorem{exe}[defin]{Example}

\input cyracc.def
\font\tencyr=wncyr10
\def\cyr{\tencyr\cyracc}
\newcommand{\ts}{\mbox{\cyr Sh}}

\title[Control theorems for $\ell$-adic Lie extensions...]{Control Theorems for $\ell$-adic Lie extensions of global function fields}

\author{A. Bandini and M. Valentino}

\begin{document}

\begin{abstract}
Let $F$ be a global function field of characteristic $p>0$, $K/F$ an $\l$-adic Lie extension
unramified outside a finite set of places $S$ and $A/F$ an abelian
variety. We study $Sel_A(K)_\l^\vee$ (the Pontrjagin dual of the Selmer
group) and (under some mild hypotheses) prove that it is a finitely generated $\Z_\l[[\Gal(K/F)]]$-module
via generalizations of Mazur's Control Theorem. If $\Gal(K/F)$ has no elements of order $\l$ and contains
a closed normal subgroup $H$ such that $\Gal(K/F)/H\simeq \Z_\l$, we are able to give sufficient conditions
for $Sel_A(K)_\l^\vee$ to be finitely generated as $\Z_\l[[H]]$-module and, consequently, a torsion
$\Z_\l[[\Gal(K/F)]]$-module. We deal with both cases $\l\neq p$ and $\l=p$.
\end{abstract}

\maketitle

\noindent{\bf MSC (2010):} 11R23 (primary), 11R58, 11G35 (secondary).

\noindent{\bf Keywords:} Selmer groups, abelian varieties, function fields, Lie extensions.

\section{Introduction}
One of the main features of Iwasawa theory is the link it provides between characteristic ideals of
(duals of) Selmer groups and $p$-adic $L$-functions. In the classical abelian setting of
$\Z_\l^d$-extensions both its analytic and algebraic sides have been well developed for general global
fields, leading to the statements (and, in some cases, to the proofs) of Main Conjectures.
The foundations for the research in non-commutative Iwasawa theory for a general $\l$-adic
Lie extension can be found in the celebrated paper by Coates et al. \cite{CFKSV} and the subject
has been developed from there, first in the number field setting and (more recently) for
function fields of positive characteristic.\\

\noindent In this work we deal with the function field setting. Namely let $F$ be a global function field
of trascendence degree one over its constant field $\F_q$, where $q$ is a power of a fixed prime $p\in \Z$.
Let $K$ be a Galois extension of $F$ unramified outside a finite set of primes $S$ and such that $G=\Gal(K/F)$
is an infinite $\l$-adic Lie group ($\l\in\Z$ a prime number). Let $\L(G)$ be the
associated Iwasawa algebra (for precise definitions of all notations and objects appearing
in this Introduction see Section \ref{SecNot}). Finally, let $A$ be an abelian variety defined over $F$
and of finite dimension $g$.\\
When $G$ has no elements of order $\l$ and contains a closed normal subgroup $H$ such
that $G/H\simeq \Z_\l\,$, Coates et al. \cite{CFKSV} are able to define a characteristic
element for every finitely generated $\L(G)$-module $M$, which has the property that the
quotient of $M$ by the submodule of elements of order a finite power of $\l$ is finitely
generated over the Iwasawa algebra of $H$ (i.e., {\em $M$ belongs to the category $\mathfrak{M}_H(G)$}
in the language of Conjecture 5.1 of \cite{CFKSV}).\\

\noindent We shall study the structure of $Sel_A(K)_\l^\vee$ (the Pontrjagin dual of the Selmer group)
as a mo\-du\-le over both $\L(G)$ and $\L(H)$ via the classical tool provided by Mazur's Control Theorem
(see \cite{Ma} and, for recent generalizations to $\Z_\l^d$-extensions of function fields, \cite{BBL},
\cite{BL}, \cite{BL2} and \cite{T}). We will prove that $Sel_A(K)_\l^\vee$
is a finitely generated $\L(G)$-module and, in a similar way, we obtain that $Sel_A(K)_\l^\vee$
is a finitely generated $\L(H)$-module as well, provided that $G$ contains a suitable closed subgroup $H$
(and, for $\l=p$, under certain hypotheses on the splitting of primes in the $\Z_\l$-extension $K^H/F$).
When this is the case, since $H$ has infinite index in $G$, $Sel_A(K)_\l^\vee$ is also a torsion $\L(G)$-module.\\
Working on the algebra $\L(H)$ is often sufficient to gather information about
$Sel_A(K)_\l^\vee$ as a $\L(G)$-module, but we prefer to work on $\L(G)$ as well for several
reasons. First of all, as one can see in \cite{CFKSV}, we always need finitely generated
$\L(G)$-modules. Secondly, in our first control theorem (Theorem \ref{CTlnotp}) it is
possible to find the path we will follow and the tools we will use in most of the proofs of the
main theorems (including the ones for $\L(H)$-modules involving the extensions $K/K^H$, which
are not, strictly speaking, control theorems since the base field will not be allowed to vary).
Finally, control theorems for $\L(G)$-modules sometimes allow us to obtain results on general
$\l$-adic Lie extensions, regardless of $G$ containing or not an arithmetically interesting subgroup $H$.\\

\noindent We recall that recent papers have considered similar and/or more general topics using
different (somehow more sophisticated) techniques. For example, in \cite{OT}, the authors study the
structure of Selmer groups using syntomic cohomology for $p$-adic Lie extensions of fields of
characteristic $p$ containing the unique $\Z_p$-extension of the constant field.
Moreover \cite{Wi} provides a comprehensive study of the whole theory
(including a proof of the main conjecture) in the language of schemes, for $\l$-adic extensions of
a separated scheme $X$ of finite type over the field $\F_{p^e}$; his proof uses K-theory, Waldhausen
categories (and higher $K$-groups) and other cohomological tools. Another approach to the main conjecture
(for $\Z_\l^d$-extensions) is provided in \cite{LLTT}.\\
Our method basically only requires a careful study of Galois cohomology groups (both for local and global
fields), hence control theorems, in addition to being interesting in their own right, provide what
can be now considered as an ``elementary approach'' to the study of Selmer groups. Nevertheless
this approach, fit also to investigate some issues of the analytic side of Iwasawa theory (as done in \cite{Pa}),
guarantees a lot of information on the structure of Selmer groups and, in particular,
many cases in which Conjecture 5.1 of \cite{CFKSV} is verified (see Corollaries \ref{FinGenH} and
\ref{ModStrl=p}).\\

\noindent Here is a summary of the paper. After recalling (in Section \ref{SecNot}) the main
objects and the setting we will work with, we shall start with the case $\l\neq p$ proving the
control theorem in Section \ref{SecCTlneqp} and examining its consequences on the $\L(G)$-structure
of $Sel_A(K)_\l^\vee$ in Section \ref{SecModStrlneqp}. In particular we shall prove the following
(see Theorem \ref{CTlnotp})

\begin{thm}
For any finite extension $F'/F$ contained in $K$, the kernel and cokernel of the map
\[ a_{K/F'}\,:\, Sel_A(F')_\l \longrightarrow Sel_A(K)_\l^{\Gal(K/F')} \]
are cofinitely generated $\Z_\l$-modules. If all primes in $S$ and all primes of bad reduction
have decomposition groups open in $G$, then the coranks of kernels and cokernels
are bounded independently of $F'\,$. Moreover if $A[\l^\infty](K)$ is finite, then
such kernels and cokernels are of finite order.
\end{thm}

\noindent An immediate application of Nakayama's Lemma will show that $Sel_A(K)_\l$ is a
cofinitely generated $\L(G)$-module. Then, in Section \ref{SecModStrlneqp}, we prove an
analogous statement (Theorem \ref{CTModlneqp}) for the kernel and cokernel of the map
\[ a_{K/K'}\,:\, Sel_A(K')_\l \longrightarrow Sel_A(K)_\l^{\Gal(K/K')} \ , \]
where $\Gal(K/K')=H$ and $G/H\simeq \Z_\l\,$. We derive from that the structure of
$Sel_A(K)_\l^\vee$ as a $\L(H)$-module and, as a consequence, some cases in which
$Sel_A(K)_\l^\vee$ is a torsion $\L(G)$-module. We also deal with the case
$K=F(A[\l^\infty])$ which can be included in this general case (with the only exception
of the ad hoc result of Theorem \ref{CTlnotpTorsion}).\\
In Sections \ref{SecCTl=p} and \ref{SecModStrl=p} we follow the same path and obtain analogous
results for the case $\l=p$, where definitions and statements need the use of flat cohomology
groups but most of the proofs only require the study of Galois cohomology groups (as in the
$\l\neq p$ case). The main difference will be the presence of nontrivial images for the Kummer
maps in the definition of the Selmer groups. In particular we shall show (see
Theorem \ref{mainthm2})

\begin{thm}
Assume that all ramified primes are of good ordinary or split multiplicative reduction,
then, for any finite extension $F'/F$ contained in $K$, the kernel and cokernel of the map
\[ a_{K/F'}\,:\, Sel_A(F')_p \longrightarrow Sel_A(K)_p^{\Gal(K/F')} \]
are cofinitely generated $\Z_p$-modules. If all primes in $S$ and all primes of bad reduction
have decomposition groups open in $G$, then the coranks of kernels and cokernels
are bounded independently of $F'\,$.
\end{thm}

\noindent As in the case $\l\neq p$ we immediately find that $Sel_A(K)_p^\vee$ is a
finitely generated $\L(G)$-module (this can be seen as the non-abelian
counterpart of \cite[Theorem 5]{T}). Moreover we will prove that $Sel_A(K)_p^\vee$ is
often $\L(G)$-torsion: when the ramified primes are of good ordinary reduction we
provide a direct proof (see Theorem \ref{Torsl=p}, which depends on the control Theorem
\ref{mainthm2}, for a precise statement), when the ramified primes are of split
multiplicative reduction we need to examine the structure of $Sel_A(K)_p^\vee$
as a $\L(H)$-module to get the result (see Corollary \ref{ModStrl=p}, which depends on the
control Theorem \ref{CTModl=p}).\\
We note that the extension $F(A[p^\infty])$ is not included here because it should require
the study of inseparable extensions which cannot be treated with the techniques used
in this paper.\\

\noindent{\bf Acknowledgments} The authors thank David Burns, Ki-Seng Tan and Otmar Venjakob
for useful discussions and comments on earlier drafts of this paper. We are grateful to John
Coates for pointing out (and providing) the thesis \cite{S}. The second named author wishes
to thank the Mathematische Institut of the University of Heidelberg for hospitality and for providing
a nice and challenging environment to start working on this project.

\section{Setting and notations}\label{SecNot}
For the convenience of the reader we recall here the main objects we shall deal with and give notations
and definitions for them. The notations will be fairly standard for Iwasawa theory so the expert reader
can simply skip this section and go back to it only when/if needed.

\subsection{Fields} For any field $L$ we let $G_L=\Gal(L^s/L)$ ($L^s$ a separable closure of $L$)
and $\X_L$ the scheme $Spec(L)$: they will essentially appear in Galois and flat cohomology groups
so whenever we write a scheme $\X$ we always mean $\X_{fl}\,$.

\noindent Let $F$ be a global function field of trascendence degree
one over its constant field $\F_F=\F_q$, where $q$ is a power of a fixed
prime $p\in \Z$. We put $\ov{F}$ for an algebraic closure of $F$ and $F^s\subset \ov{F}$
for a separable closure. \\
For any algebraic extension $L/F$ let $\M_L$ be the set
of places of $L$: for any $v\in \M_L$ we let $L_v$ be the completion of $L$ at $v$, $\mathcal{O}_v$
its ring of integers with maximal ideal $\mathfrak{m}_v$ and residue field $\mathbb{F}_{L_v}\,$.
Whenever we deal with a local field $E$ (or an algebraic extensions of such field) the above notations
will often be replaced by $\mathcal{O}_E\,$, $\mathfrak{m}_E$ and $\mathbb{F}_E\,$. \\
For any place $v\in \M_F$ we choose (and fix) an embedding $\ov{F} \iri \ov{F_v}$ (an algebraic closure
of $F_v\,$), in order to get a restriction map $G_{F_v}:=\Gal(\ov{F_v}/F_v) \iri G_F\,$. All
algebraic extensions of $F$ (resp. of $F_v\,$) will be assumed to be contained in $\ov{F}$ (resp.
$\ov{F_v}\,$).\\
In general we will deal with {\em $\l$-adic Lie extensions} $K/F$, i.e.,  Galois
extensions with Galois group an $\l$-adic Lie group. We always assume that our extensions
are unramified outside a finite set $S$ of primes of $\M_F\,$.

\subsection{$\l$-adic Lie groups} The most useful (for us) characterization of $\l$-adic analytic groups
is due to Lazard \cite{L} and states that: {\em a topological group $G$ has the structure of $\l$-adic
Lie group if and only if $G$ contains an open subgroup which is a uniform pro-$\l$ group} (see also
\cite[Theorem 8.32]{DdSMS}).

\noindent For the whole paper our $\l$-adic Lie group $G$ is the Galois group of a field
extension, so one must take into account that it is also compact and profinite. Compactness implies it has
finite rank and has an open, normal, uniform pro-$\l$ subgroup which is always finitely generated
(see \cite[Corollary 8.34]{DdSMS}). While being profinite means that every open subgroup is of finite index.

\noindent Finally, note that if $G$ is an $\l$-adic Lie group without points of order $\l$, then it has finite
$\l$-cohomological dimension, which is equal to its dimension as an $\l$-adic Lie group
(\cite[Corollaire (1) p. 413]{Se2}).

\noindent For all basic definitions and facts about profinite groups, the reader is referred to
\cite{DdSMS} and \cite{W}.

\subsection{Modules and duals} For any $\l$-adic Lie group $G$ we denote by
\[ \L(G) = \Z_\l[[G]] := \il{U} \Z_\l [G/U] \]
the associated {\em Iwasawa algebra} (the limit is on the open normal subgroups of $G$).
From Lazard's work (see \cite{L}), we know that $\L(G)$ is Noetherian and, if $G$ is pro-$\l$ and has no
elements of order $\l$, then $\L(G)$ is an integral domain. From \cite[Theorem 4.5]{DdSMS} we know
that, for a finitely generated powerful pro-$\l$ group, being torsion free is equivalent to being
uniform. Because of this when we need $\L(G)$ to be without zero divisors we will take a torsion free $G$.\\
For a $\L(G)$-module $M$, we denote its Pontrjagin dual by $M^{\vee}:=Hom_{cont}(M,\mathbb{C}^{\ast})$.
In the cases considered in this paper, $M$ will be a (mostly discrete) topological $\Z_\l$-module,
so that $M^{\vee}$ can be identified with $Hom_{cont}(M,\Q_\l/\Z_\l)$ and it has a natural structure of
$\Z_\l$-module (because the category of compact $\L(G)$-modules and the category of discrete
$\L(G)$-modules are both abelian and the Pontrjagin duality defines a contravariant equivalence
of categories between them).\\
The reader is reminded that to say that an $R$-module $M$ ($R$ any ring) is
{\em cofinitely generated over} $R$ means that $M^{\vee}$ is a finitely generated $R$-module.

\subsection{Selmer groups} Let $A$ be an abelian variety of dimension $g$ defined over $F$:
we denote by $B$ its dual abelian variety. For any positive integer
$n$ we let $A[n]$ be the scheme of $n$-torsion points and, for any prime $\l$, we put
$A[\l^\infty]:={\displaystyle \lim_{\rightarrow}\,} A[\l^n]\,$.\\

\noindent We define Selmer groups via the usual cohomological techniques and, since we deal mostly with
the flat scheme of torsion points, we shall use the flat cohomology groups $H^i_{fl}$ (for the
basic definitions see \cite[Ch. II and III]{Mi2}). Fix a prime $\l\in\Z$ and consider the exact
sequence
\[ 0 \rightarrow A[\l^n] \rightarrow A {\buildrel \ \l^n \over \longrightarrow} A \rightarrow 0 \]
and, for any finite algebraic extension $L/F$, take flat cohomology with respect to $\X_L$ to get
an injective {\em Kummer map}
\[ A(L)/\l^n A(L) \iri H^1_{fl}(\X_L\,,A[\l^n]) \ .\]
Taking direct limits one has an injective map
\[ \k : A(L) \otimes \Q_\l/\Z_\l \iri \dl{n} H^1_{fl}(\X_L\,,A[\l^n]):=H^1_{fl}(\X_L\,,A[\l^\infty]) \ .\]
Exactly in the same way one can define {\em local Kummer maps}
\[  \k_w : A(L_w) \otimes \Q_\l/\Z_\l \iri \dl{n} H^1_{fl}(\X_{L_w}\,,A[\l^n]):=
H^1_{fl}(\X_{L_w}\,,A[\l^\infty]) \]
for any place $w\in\M_L\,$.

\begin{defin}\label{DefSel}
{\em The} $\l$-part of the Selmer group {\em of $A$ over $L$ is defined to be
\[ Sel_A(L)_\l=Ker \left\{ H_{fl}^1(\X_L, A[\l^{\infty}]) \to \prod_{w\in\M_L}
H_{fl}^1(\X_{L_w}, A[\l^{\infty}])/Im\,\k_w \right\}\]
where the map is the product of the natural restrictions between cohomology groups.\\
For infinite extensions $\mathcal{L}/F$ the Selmer group $Sel_A(\mathcal{L})_\l$
is defined, as usual, via direct limits.}
\end{defin}

\noindent Letting $L$ vary through subextensions of $K/F$, the groups $Sel_A(L)_\l$ admit natural actions
by $\mathbb{Z}_\l$ (because of $A[\l^\infty]\,$) and by $G=\Gal(K/F)$. Hence they are modules over
the Iwasawa algebra $\L(G)$.

\noindent If $L/F$ is a finite extension the group $Sel_A(L)_\l$ is a cofinitely generated $\Z_\l$-module
(see, e.g. \cite[III.8 and III.9]{Mi1}). One can define the {\em Tate-Shafarevich group} $\ts(A/L)$
as the group that fits into the exact sequence
\[ A(L)\otimes \Q_\l/\Z_\l \iri Sel_A(L)_\l \sri \ts(A/L)[\l^\infty]\ .\]
The function field version of the Birch and Swinnerton-Dyer conjecture predicts $\ts(A/L)$ to be finite
for any finite extension $L$ of $F$ (relevant results in this direction can be found,
for example, in \cite{KT} and \cite{Sc}).
Assuming the BSD conjecture and taking Pontrjagin duals on the sequence above, one gets
\[ \rank_{\Z_\l}\,Sel_A(L)_\l^{\vee} = \rank_{\Z}\,A(L) \]
which provides motivation for the study of (duals of) Selmer groups (recall that the cohomology
groups $H_{fl}^i\,$, hence the Selmer groups, are endowed with the discrete topology).

\begin{rem}
{\em When $\l\neq p$ the torsion subschemes are Galois modules and we can define Selmer groups
via Galois cohomology since, in this case,
\[ H_{fl}^1(\X_L,A[\l^n])\simeq H_{et}^1(\X_L,A[\l^n])\simeq H^1(G_L,A[\l^n](\ov{F})) \]
(see \cite[III.3.9]{Mi2}). In order to lighten notations, each time we work with
$l\neq p$ we shall use the classical notation $H^i(L,\cdot)$ instead of $H^i(G_L,\cdot)$
and $H^i(L/E,\cdot)$ instead of $H^i(\Gal(L/E),\cdot)$. Moreover we write $A[n]$ for $A[n](\ov{F})$,
putting $A[\l^{\infty}]:=\bigcup A[\l^n]$. }
\end{rem}

\subsection{The fundamental diagram} We will consider $\l$-adic Lie extensions $K/F'$ (for any
finite extension $F'/F$) and study the kernels and cokernels of the natural restriction maps
\[ Sel_A(F')_\l \longrightarrow Sel_A(K)_\l^{\Gal(K/F')} \ .\]
We shall do this via the snake lemma applied to the following diagram
\begin{equation}\label{FondDiag}
\xymatrix{
 Sel_A(F')_\l \ar[d]^{a_{K/F'}} \ar@{^{(}->}[r] & H^1_{fl}(\X_{F'},A[\l^{\infty}]) \ar[d]^{b_{K/F'}} \ar@{->>}[r]
& \G_A(F') \ar[d]^{c_{K/F'}}\\
Sel_A(K)^{\Gal(K/F')}_\l \ar@{^{(}->}[r]
& H^1_{fl}(\X_K,A[\l^{\infty}])^{\Gal(K/F')} \ar[r] & \G_A(K)^{\Gal(K/F')} &  , }
\end{equation}
where, for any field $L$, we put
\[ \G_A(L)=Im \left\{ H^1_{fl}(\X_L, A[\l^{\infty}]) \to
\prod_{w\in\M_L} H^1_{fl}(\X_{L_w}, A[\l^{\infty}])/Im\,\k_w \right\}\ .\]

\section {Control theorem for $\l\neq p$.}\label{SecCTlneqp}
As pointed out by Coates and Greenberg in \cite[Proposition 4.1 and the subsequent Remark]{CG}, the image of
the Kummer map $\k_w$ is trivial for any $w\in\M_L$,
because $L_w$ has characteristic $p\neq \l$. Therefore, in this case,
the $\l$-part of the Selmer group is simply
\[ Sel_A(L)_\l=Ker \left\{ H^1(L, A[\l^{\infty}]) \to \prod_{w\in\M_L} H^1(L_w, A[\l^{\infty}]) \right\} \ .\]

The following lemma holds for any prime $\l$ (including $\l=p$) and we shall need it in most of our control theorems.

\begin{lem}\label{CohoBound}
Let $G$ be any compact $\l$-adic Lie group of finite dimension and let $M$ be a discrete $\L(G)$-module which
is cofinitely generated over $\Z_\l$ (resp. finite). Then, for any closed subgroup $\V$ of $G$,
the cohomology groups $H^1(\V,M)$ and $H^2(\V,M)$ are cofinitely generated (resp. finite) $\Z_\l$-modules as well
and their coranks (resp. their orders) are bounded independently of $\V$.
\end{lem}

\begin{proof}
Obviously we can restrict our attention to pro-$\l$ subgroups of $G$ and, by \cite[Ch. 1, Exercise 12]{DdSMS},
any group of this kind is contained in a pro-$\l$ Sylow subgroup $P_\l\,$. Moreover any pro-$\l$ Sylow is open by
\cite[Corollary 8.34]{DdSMS}. Let $\V$ be any closed pro-$\l$ subgroup of some pro-$\l$ Sylow $P_\l$ and put
\[ d_i(\V):= \dim_{\Z/\l\Z}\, H^i(\V,\Z/\l\Z)\qquad i=1,2\ .\]
Since $P_\l$ has finite rank (in the sense of \cite[Definition 3.12]{DdSMS}) the cardinalities of a minimal
set of topological generators for $\V$, i.e., the $d_1(\V)$'s, are all finite and bounded by $d_1(P_\l)$.
Moreover, since $P_\l$ contains a uniform (see \cite[Definition 4.1]{DdSMS}) open subgroup $\mathcal{U}$
(by \cite[Corollary 8.34]{DdSMS}), one has
\[ d_1(\V) \leqslant d_1(\mathcal{U}) = d := {\rm dimension\ of\ } G \]
for all closed subgroups $\V$ of $\mathcal{U}$. For $d_2(\V)$ (i.e., the numbers of relations for a
minimal set of topological generators of $\V$) the bound is provided by \cite[Theorem 4.35]{DdSMS}
(see also \cite[Ch. 4, Exercise 11]{DdSMS}) again only in terms of the dimension and rank of $G$
(for example, if $\mathcal{U}$ is as above, then $d_2(\mathcal{U}) = \frac{d(d-1)}{2}\,$).
We put $\tilde{d}_1=\tilde{d}_1(G)$ (resp. $\tilde{d}_2=\tilde{d}_2(G)\,$) as the upper bound for
all the $d_1(\V)$'s (resp. $d_2(\V)$'s) as $\V$ varies among the closed subgroups of any pro-$\l$
Sylow of $G$.

\noindent Let $M_{div}$ be the maximal divisible subgroup of $M$ and consider the finite quotient $M/M_{div}\,$.
By \cite[Ch. I, \S 4, Proposition 20]{SeCG} and the Corollaire right after it, the finite group
$M/M_{div}$ admits a $\V$-composition series with quotients isomorphic to $\Z/\l\Z$.
Hence, working as in \cite[Proposition 3.1]{G}, one immediately finds
\[ |H^i(\V,M/M_{div})|\leqslant |M/M_{div}|^{d_i(\V)} \leqslant |M/M_{div}|^{\tilde{d}_i} \]
for $i=1,2$.

\noindent Note that, if $M$ is finite (hence equal to $M/M_{div}\,$),
we have already completed the proof of the finiteness of the $H^i(\V,M)$.\\
Moreover, in this case the orders are also bounded independently of $\V\,$.\\

\noindent To deal with the divisible part note that $M_{div}[\l]$ is finite
(say of order $\l^\lambda\,$), thus, by what we have just proved,
\[ |H^i(\V,M_{div}[\l])| \leqslant \l^{\lambda\tilde{d}_i} \quad (i=1,2)\ .\]
The cohomology of the exact sequence
\[ M_{div}[\l] \iri M_{div} \buildrel \l \over \sri M_{div} \]
yields surjective morphisms
\[  H^i(\V,M_{div}[\l]) \sri H^i(\V,M_{div})[\l] \quad (i=1,2)\ ,\]
therefore the groups on the right have finite and bounded orders as well.\\
Hence the $H^i(\V,M_{div})$ are cofinitely generated $\Z_\l$-modules with coranks bounded by
\[ \tilde{d}_i\lambda = \tilde{d}_i\corank_{\Z_\l} M \]
(which is a bound for the whole $H^i(\V,M)$ since the $\Z_\l$-corank of the finite part is 0).
\end{proof}

\begin{thm}\label{CTlnotp}
With the above notations, for every finite extension $F'$ of $F$ contained in $K$, the kernels and
cokernels of the maps
\[ a_{K/F'}: Sel_A(F')_\l \to Sel_A(K)^{\Gal(K/F')}_\l \]
are cofinitely generated $\Z_\l$-modules (i.e., their Pontrjagin duals are finitely generated
$\Z_\l$-modules). If all primes in $S$ and all primes of bad reduction
have decomposition groups open in $G$, then the coranks of kernels and cokernels
are bounded independently of $F'\,$.
Moreover if $A[\l^\infty](K)$ is finite, then such kernels and cokernels are of finite order.
\end{thm}

\begin{proof}
We shall use the snake lemma for the diagram \eqref{FondDiag}: hence, to prove the theorem, we are going
to bound the kernels and cokernels of the maps $b_{K/F'}$ and $c_{K/F'}$.\\

\subsection{The maps $b_{K/F'}$}\label{Mapblnotp}
From the Hochschild-Serre spectral sequence we have that
\[ Ker (b_{K/F'})\simeq H^1(K/F', A[\l^\infty](K))\quad\mathrm{and}
\quad Coker(b_{K/F'})\subseteq H^2(K/F', A[\l^\infty](K)) \ .\]
We can apply Lemma \ref{CohoBound} with $G=\Gal(K/F)$ and $M=A[\l^\infty](K)$.
Hence $Ker (b_{K/F'})$ and $Coker (b_{K/F'})$ are cofinitely generated $\Z_\l$-modules whose coranks are bounded by
\[ \tilde{d}_1(\Gal(K/F))\corank_{\Z_\l} A[\l^\infty](K)\quad{\rm and}\quad
\tilde{d}_2(\Gal(K/F))\corank_{\Z_\l} A[\l^\infty](K) \]
respectively. Moreover, if $A[\l^\infty](K)$ is finite, we have the bound
\[ |H^i(\V,A[\l^\infty](K))|\leqslant |A[\l^\infty](K)|^{\tilde{d}_i(\Gal(K/F))} \]
(for $i=1,2$).

\subsection{The maps $c_{K/F'}$}\label{Mapclnotp}
For every prime $v\in\M_F$ let $v'$ be a place of $F'$ lying above $v$. Observe that
\[ Ker(c_{K/F'})\iri \prod_{v'\in\M_{F'}}
\bigcap_{\begin{subarray}{c} w \in \M_K \\ w|v' \end{subarray} }\,Ker(d_w)\]
where
\[d_w: H^1(F'_{v'},A[\l^{\infty}])\to H^1(K_w,A[\l^{\infty}]) \]
and, by the Inf-Res sequence,
\[ Ker(d_w) = H^1(K_w/F'_{v'},A[\l^\infty](K_w)) \ .\]
From the Kummer exact sequence one can write the following diagram
\[ \xymatrix{\ar @{} [dr]
 H^1(F'_{v'},A[\l^{\infty}]) \ar[d]^{d_w} \ar@{^{(}->}[r] &
H^1(F'_{v'},A(F_{v}^s)) \ar[d]^{f_w} \\
H^1(K_w,A[\l^{\infty}]) \ar@{^{(}->}[r] & H^1(K_w,A(F_v^s)) } \]
and deduce from it the inclusion
\[ Ker(d_w)\iri Ker(f_w)\simeq H^1(K_w/F'_{v'}, A(K_w))\ . \]
Since this last group is obviously trivial if the prime $v'$ splits completely we limit ourselves to the
study of this local kernels for primes which are not totally split in $K/F'\,$.

\noindent If $v$ is a prime of good reduction for $A$ and unramified in $K/F'$, then
\cite[Ch. I, Proposition 3.8]{Mi1} yields \[ H^1((F'_{v'})^{unr}/F'_{v'},A((F'_{v'})^{unr}))=0 \ . \]
Via the inflation map one immediately gets $H^1( K_w/F'_{v'}, A(K_w))=0$.

\noindent Thus we are left with
\[ Ker(c_{K/F'})\iri \prod_{\begin{subarray}{c} v'\in\M_{F'} \\ v'|v\in S^* \end{subarray}}
\bigcap_{\begin{subarray}{c} w \in \M_K \\ w|v'\end{subarray}}\,Ker(d_w) \]
where $S^*$ is the finite set composed by:
\begin{itemize}
\item[-] all primes in $S$;
\item[-] all primes not in $S$ and of bad reduction for $A$.
\end{itemize}
To find bounds for these primes we shall use Tate's theorems on (local) duality and (local)
Euler-Poincar\'e characteristic.\\

\noindent By \cite[Theorem 7.1.8]{NSW}, the group $H^1(F'_{v'}, A[\l^\infty])$ is finite and we can
bound its order using the Euler characteristic
\[ \chi(F'_{v'},A[\l^\infty]):=
\frac{|H^0(F'_{v'}, A[\l^\infty])||H^2(F'_{v'}, A[\l^\infty])|}{|H^1(F'_{v'}, A[\l^\infty])|} \ .\]
From the pairing on cohomology induced by the Weil pairing (see, for example, \cite[Ch. I, Remark 3.5]{Mi1}),
the group $H^2(F'_{v'}, A[\l^\infty])$ is the (Pontrjagin) dual of $H^0(F'_{v'}, B[\l^\infty])$,
so all the orders in the formula are finite. Moreover, by \cite[Theorem 7.3.1]{NSW},
$\chi(F'_{v'},A[\l^\infty])=1$ \footnote{The statement is for finite modules (i.e., it's true for all
the modules $A[\l^n]$ for any $n$), but limits are allowed here since the numerator stabilizes,
i.e., there is an $n$ such that
\[ |A[\l^m](F'_{v'})||B[\l^m](F'_{v'})|=|A[\l^n](F'_{v'})||B[\l^n](F'_{v'})| \]
for all $m\ge n.$}, therefore
\[ |H^1(F'_{v'}, A[\l^\infty])|=|A[\l^\infty](F'_{v'})||B[\l^\infty](F'_{v'})| \ .\]
The inflation map
\[ H^1(K_w/F'_{v'},A[\l^\infty](K_w)) \iri H^1(F'_{v'}, A[\l^\infty]) \]
provides again the finiteness of the kernels (but note that here, in general, the orders are not bounded).
\end{proof}

\begin{rem}\ \begin{itemize}
\item[{\bf 1.}] {\em Note that the local kernels are always finite, the additional hypothesis on the finiteness
of $A[\l^\infty](K)$ was only used to bound the orders of $Ker(b_{K/F'})$ and $Coker(b_{K/F'})$.}
\item[{\bf 2.}] {\em Most of the bounds are independent of $F'$ (in particular the ones for $Ker(a_{K/F'})\,$),
but to get uniform bounds for $Coker(a_{K/F'})$ one also needs to bound the number of nontrivial
groups appearing in the product which contains $Ker(c_{K/F'})$. In particular one needs finitely
many nontrivial $Ker(d_w)$'s and the only way to get this is assuming that the decomposition groups
are open. Note that if there is at least one unramified prime of bad reduction such hypothesis on its
decomposition group immediately yields that $G$ contains a subgroup isomorphic to $\Z_\l$ with finite index.}
\end{itemize}
\end{rem}

We can be a bit more precise with the local bounds for the unramified primes of bad reduction thanks
to the following

\begin{prop}\label{UnrPrilnotp}
If $v$ is an unramified prime of bad reduction (with nontrivial decomposition group)
and $v'|v$, then $H^1(K_w/F'_{v'},A[\l^\infty](K_w))$ is finite and
its order is bounded independently of $F'\,$.
\end{prop}

\begin{proof}
The $\l$-part of the Galois group $\Gal(K_w/F'_{v'})$ is a finite cyclic $\l$-group or it is isomorphic to
$\Z_\l\,$. Moreover
\[ A[\l^\infty](K_w)^{\Gal(K_w/F'_{v'})} = A[\l^\infty](F'_{v'}) \]
is finite, so, in the $\Z_\l$ case (i.e., when $K_w$ contains $(F'_{v'})^{\l,unr}$ the unramified
$\Z_\l$-extension of $F'_{v'}\,$), we can apply \cite[Remark 3.5]{BL2} to get
\[ |H^1(K_w/F'_{v'},A[\l^\infty](K_w))| \leqslant |A[\l^\infty](K_w)/A[\l^\infty](K_w)_{div}| \ .\]
A similar bound (independent from $F'\,$) holds for the finite case as well. One just uses the
inflation map to the group $H^1(K_w(F'_{v'})^{\l,unr}/F'_{v'}, A[\l^\infty](K_w(F'_{v'})^{\l,unr}))$
which has order bounded by $|A[\l^\infty](K_w(F_{v})^{\l,unr})/A[\l^\infty](K_w(F_{v})^{\l,unr})_{div}|$.
\end{proof}

We summarize the given bounds with the following

\begin{cor}
In the setting of Theorem \ref{CTlnotp} assume that all primes in $S$ and all primes of bad reduction
have decomposition groups open in $G$, then one has:\begin{itemize}
\item[{\bf 1.}] $\corank_{\Z_\l} Ker(a_{K/F'}) \leqslant \tilde{d}_1\corank_{\Z_\l} A[\l^\infty](K)$ and \\
$\corank_{\Z_\l} Coker(a_{K/F'}) \leqslant \tilde{d}_2\corank_{\Z_\l} A[\l^\infty](K)\,$;
\item[{\bf 2.}] if $A[\l^\infty](K)$ is finite, then $|Ker(a_{K/F'})| \leqslant |A[\l^\infty](K)|^{\tilde{d}_1}$ and
\[ |Coker(a_{K/F'})| \leqslant |A[\l^\infty](K)|^{\tilde{d}_2} \prod_{v'|v\in S^*-S}\,\alpha_v
\prod_{v'|v\in S}\, \beta_{v'}  \]
(where
\[ \alpha_v = |A[\l^\infty](K_w(F_{v})^{\l,unr})/A[\l^\infty](K_w(F_{v})^{\l,unr})_{div}| \]
and
\[ \beta_{v'} = |A[\l^\infty](F'_{v'})||B[\l^\infty](F'_{v'})|\ )\ ;\]
\item[{\bf 3.}] if $A[\l^\infty](K)$ is finite and, for all primes $v\in S$, $A[\l^\infty](K_w)$ is finite for any
prime $w|v'|v$, then $|Ker(a_{K/F'})| \leqslant |A[\l^\infty](K)|^{\tilde{d}_1}$ and
\[ |Coker(a_{K/F'})| \leqslant  |A[\l^\infty](K)|^{\tilde{d}_2} \prod_{v'|v\in S^*-S}\,\alpha_v
\prod_{v'|v\in S}\, |A[\l^\infty](K_w)||B[\l^\infty](K_w)|\ . \]
\end{itemize}
The bounds in {\bf 1} are independent of $F'$ (while the values appearing in {\bf 3} do not depend on $F'$
but the number of the factors in the product does).
\end{cor}

\subsection{The case $K=F(A[\l^\infty])$}\label{TorExt}
The finiteness of $A[\l^\infty](K)$ is not a necessary condition to get finite kernels and cokernels in the
control theorem. An important example is provided by the extension $K=F(A[\l^\infty])$ (this extensions has
been studied in details in \cite{S} in the case $A=E$, an elliptic curve).
Fix a basis of $A[\l^\infty]$ and consider the continuous Galois representation
\[ \rho:\ \Gal(F^s/F)\to GL_{2g}(\Z_\l) \]
provided by the action of $\Gal(F^s/F)$ on the chosen basis. Since $Ker(\rho)$ is given by the automorphisms
which fix $A[\l^\infty]$, one gets an isomorphism
\[ \Gal(K/F) \simeq  \rho(\Gal(F^s/F)) \]
and, consequently, an embedding
\[ \Gal(K/F)\hookrightarrow GL_{2g}(\Z_\l) \ .\]
Since $\Gal(F^s/F)$ is compact, its image under $\rho$ must be a compact subgroup of $GL_{2g}(\Z_\l)$.
As the latter is Hausdorff, $\rho(\Gal(F^s/F))$ is closed and so it is an $\l$-adic Lie group.

\begin{rem}
{\em When $A$ has genus 1 (i.e., it is an elliptic curve), a theorem of Igusa (analogous to Serre's
open image theorem) gives a more precise description of $\Gal(K/F)$ (for a precise statement and a proof see
\cite{BLV} and the references there). For a general abelian variety such {\em open image} statements are not
known: some results in this direction for abelian varieties of ``Hall type'' can be found in \cite{Ha} and
\cite{AGP} }.
\end{rem}

\begin{thm}\label{CTlnotpTorsion}
Let $K=F(A[\l^\infty])$, then the kernels and cokernels of the maps
\[ a_{K/F'}: Sel_A(F')_\l \to Sel_A(K)^{\Gal(K/F')}_\l \]
are finite.
\end{thm}

\begin{proof}
Observe that thanks to \cite[Corollary 2 (b), p. 497]{ST} only primes of bad reduction for $A$ are ramified
in the extension $F(A[\l^\infty])/F$; so, in this case, the set $S$ is obviously finite and there are no
unramified primes of bad reduction.\\
The bounds for $Ker(c_{K/F'})$ are already in Theorem \ref{CTlnotp},
so we only need to provide bounds for $Ker(b_{K/F'})$ and $Coker(b_{K/F'})$, i.e., for $H^1(K/F',A[\l^\infty])$
and $H^2(K/F',A[\l^\infty])$ (note that here $A[\l^\infty](K)=A[\l^\infty]\,$).

\noindent Consider the Tate module
\[ T_\l(A) := \il n A[\l^n]\quad{\rm and\ put}\quad V_\l(A) := T_\l(A)\otimes_{\Z_\l} \Q_\l \ ,\]
then one has an exact sequence (of Galois modules)
\[ T_\l(A) \iri V_\l(A) \sri A[\l^\infty] \ .\]
By \cite[Th\'eor\`eme 2]{Se3} and the subsequent Corollaire (which hold in our setting as well, as noted
in the Remarques following the Corollaire), one has $H^i(K/F',V_\l(A))=0$ for any $i\geqslant 0$, moreover
the $H^i(K/F',T_\l(A))$ are all finite groups. Hence we get isomorphism
\[ H^i(K/F', T_\l(A))\simeq H^{i-1}(K/F', A[\l^\infty]) \quad{\rm for\ any}\ i\geqslant 1 \]
which provide the finiteness of $Ker(b_{K/F'})$ and $Coker(b_{K/F'})$.
\end{proof}

\section{$\L$-modules for $\l\neq p$}\label{SecModStrlneqp}
In this section we assume that our Galois group $G$ (still an $\l$-adic Lie group) has no elements
of order $\l$ and write $\L(G)$ for the associated Iwasawa algebra.
First we describe the structure of $Sel_A(K)_\l^\vee$ as a $\L(G)$-module, showing that it is a finitely
generated (sometimes torsion) $\L(G)$-module. Then, assuming that $G$ contains a subgroup $H$ such that
$G/H\simeq \Z_\l\,$, we will show that $Sel_A(K)_\l^\vee$ is finitely generated
as a $\L(H)$-module as well. For the latter we shall need a slightly modified version of the previous
Theorem \ref{CTlnotp}. As usual the main tool for the proof (along with the control theorem) is
the following generalization of Nakayama's Lemma.

\begin{thm}\label{BHNak}
Let $G$ be a topologically finitely generated, powerful and pro-$\l$ group and $I$
any proper ideal. Let $M$ be a compact $\L(G)$-module, then:
\begin{itemize}
\item[{\bf 1.}] if $M/IM$ is finitely generated as
a $\L(G)/I$-module then $M$ is finitely generated as a $\L(G)$-module;
\item[{\bf 2.}] if $G$ is soluble uniform and $M/I_G M$ is finite then $M$ is a torsion
$\L(G)$-module (where $I_G = Ker(\L(G)\to \Z_\l)$ is the augmentation ideal of $\L(G)\,$).
\end{itemize}
\end{thm}

\begin{proof}
See the main results of \cite{BH}.
\end{proof}

\noindent Let $\F_p^{(\l)}$ be the $\Z_\l$-extension of $\F_p\,$. One of the most (arithmetically)
interesting example is provided by extensions $K/F$ containing $\F_p^{(\l)}F$, where one can take
$H=\Gal(K/\F_p^{(\l)}F)$. This can be considered as a very general setting thanks to the
following lemma (which was brought to our attention by David Burns).

\begin{lem}\label{BurnsVen}
Let $K/F$ be a $\l$-adic Lie extension with $\l\neq 2,p\,$. Then there exist a field $K'\supseteq K$ such that
\begin{itemize}
\item[{\bf 1.}] $K'$ contains $\F_p^{(\l)}F$;
\item[{\bf 2.}] $K'/F$ is unramified outside a finite set of places;
\item[{\bf 3.}] $\Gal(K'/F)$ is a compact $\l$-adic Lie group without elements of order $\l$.
\end{itemize}
\end{lem}

\begin{proof} Since $\l\neq p$ the proof is the same of \cite[Lemma 6.1]{BuVe} (where the statement is
for number fields). One basically considers the field $K(\boldsymbol{\mu}_{\l^\infty})$ (which
obviously contains $\F_p^{(\l)}F$), where $\boldsymbol{\mu}_{\l^\infty}$ is the set of all $\l$-power
roots of unity, and then cuts out elements of order $\l$ in
$\Gal(K(\boldsymbol{\mu}_{\l^\infty})/F(\boldsymbol{\mu}_{\l^\infty}))$ using Kummer theory to describe
generators for subextensions of degree $\l$.
\end{proof}

\subsection{Structure of $Sel_A(K)_\l^\vee$ as $\L(G)$-module}
We are now ready to prove the following

\begin{thm}\label{fingenlneqp}
In the setting of Theorem \ref{CTlnotp}, $Sel_A(K)_\l^\vee$ is a
finitely generated $\L(G)$-module.
\end{thm}

\begin{proof} Consider any open, powerful and pro-$\l$ subgroup $G'$ of $G$. Since $\L(G)$ is
finitely ge\-ne\-ra\-ted over $\L(G')$ it is obvious that $Sel_A(K)_\l^\vee$ is finitely generated
over $\L(G)$ if and only if it is finitely generated also over $\L(G')$. So we are going to prove
the statement for such $G'$.\\
Consider the exact sequence
\begin{equation}\label{SuccFinGen}
0\to Coker (a_{K/F'})^\vee\to (Sel_A(K)_\l^{G'})^\vee\to Sel_A(F')_\l^\vee\to Ker(a_{K/F'})^\vee \to 0
\end{equation}
where $F'$ is the fixed field of $G'$. We know from Theorem \ref{CTlnotp} that  $Coker (a_{K/F'})^\vee$
and $Ker(a_{K/F'})^\vee$ are finitely generated $\Z_\l$-modules. Moreover, since $F'/F$ is finite,
$Sel_A(F')_\l^\vee$ is a finitely generated $\Z_\l$-module (\cite[III.8 and III.9]{Mi1}).
Hence $(Sel_A(K)_\l^{G'})^\vee$ is a finitely generated $\Z_\l$-module thanks to the exactness
of the sequence \eqref{SuccFinGen}. Since $(Sel_A(K)_\l^{G'})^\vee$ is isomorphic to
$Sel_A(K)_\l^\vee/I_{G'}Sel_A(K)_\l^\vee$ , where $I_{G'}$ is the augmentation ideal,
our claim follows from Theorem \ref{BHNak}.
\end{proof}

\begin{rem}
{\em Note that in the above proof we do not need any hypothesis on the elements
of $G$ of order $\l$: it works in general for any compact $\l$-adic Lie group $G$.
That additional hypothesis is necessary only to prove that $Sel_A(K)_\l^\vee$ is a
torsion module, because we need to avoid zero divisors.}
\end{rem}

\begin{thm}\label{TorGModlneqp}
Suppose that there exists an open uniform, pro-$\l$ and soluble subgroup $G'$ of $G$, with fixed field $F'\,$.
Assume that $A[\l^\infty](K)$ and $Sel_A(F')_\l^\vee$ are finite. Then $Sel_A(K)_\l^\vee$
is a torsion $\L(G')$-module.
\end{thm}

\begin{proof}
Just use Theorems \ref{CTlnotp} and \ref{BHNak}.
\end{proof}

\begin{rem}
{\em The hypothesis on the existence of the soluble subgroup $G'$ is necessary. Indeed, when $G'$ is not soluble
it is possible to find a non torsion ideal $J$ of $\L(G')$ such that $J/I_{G'}J$ is finite (see \cite[p. 228]{BH}).
However, we observe that when $G$ is finitely generated (not only ``topologically'' finitely generated) such an open soluble
subgroup $G'$ always exists (see \cite{LMS}).}
\end{rem}

\noindent In the context of non-commutative Iwasawa algebras the right definition of {\em torsion module}
(\cite[Definition 2.6]{Vj}) can be stated in the following way: {\em a finitely generated $\L(G)$-module $M$
is a $\L(G)$-torsion module if and only if $M$ is a $\L(G')$-torsion module (classical meaning) for some
open pro-$\l$ subgroup $G'\subseteq G$ such that $\L(G')$ is integral}. So Theorem \ref{TorGModlneqp}
immediately yields

\begin{cor}\label{CTCorTor}
Let $G$ be without elements of order $\l$ and suppose that there exists an open, uniform, pro-$\l$ and soluble subgroup $G'$
of $G$. If $A[\l^\infty](K)$ and $Sel_A(F')_\l^\vee$ (where $F'$ is the fixed field of $G'\,$) are finite, then
$Sel_A(K)_\l^\vee$ is a torsion $\L(G)$-module.
\end{cor}

\subsection{Structure of $Sel_A(K)_\l^\vee$ as $\L(H)$-module}\label{Hmodlneqp}
Assume that $G$ contains a closed normal subgroup $H$ such that $G/H=\Gamma\simeq \Z_\l$.
We are going to prove that $Sel_A(K)_\l^\vee$ is a finitely generated $\L(H)$-module, as predicted by
Conjecture 5.1 of \cite{CFKSV}. First note that, letting $\F_p^{(\l)}$ be the unique $\Z_\l$-extension of $\F_p\,$,
by \cite[Proposition 4.3]{BL2}, one has $K':=K^H=\F_p^{(\l)}F$. Hence all primes of $F$ are unramified in
$K'$ and none of them is totally split.

\noindent As mentioned before we need to prove a slightly modified version of the Control Theorem.
We will work with the following diagram
\begin{equation}\label{DiagModCT}\xymatrix{\ar @{} [dr]
 Sel_A(K')_\l \ar[d]^{a} \ar@{^{(}->}[r] & H^1(K',A[\l^{\infty}]) \ar[d]^{b} \ar@{->>}[r] & \G_A(K') \ar[d]^{c}\\
Sel_A(K)^{H}_\l \ar@{^{(}->}[r]
& H^1(K,A[\l^{\infty}])^{H} \ar[r] & \G_A(K)^{H}} \end{equation}
similar to the diagram \eqref{FondDiag} except for the ``infinite level'' of the upper row, and we will again
apply the snake lemma.

\begin{thm}\label{CTModlneqp}
With the above notations the kernel and cokernel of the map
\[ a: Sel_A(K')_\l \to Sel_A(K)^{H}_\l \]
are cofinitely generated $\Z_\l$-modules. Moreover, if $A[\l^\infty](K)$ is finite and,
for any $w|w'|v\in S$, the group $A[\l^\infty](K_w)$ is finite as well, then $Ker(a)$ and
$Coker(a)$ are finite.
\end{thm}

\begin{proof}
As usual we are going to work on kernels and cokernels of the maps $b$ and $c$ in the diagram \eqref{DiagModCT}.

\subsection{The map $b$} From the Hochschild-Serre spectral sequence we have that
\[ Ker (b)\simeq H^1(K/K', A[\l^\infty](K))\quad\mathrm{and}\quad Coker(b)\subseteq H^2(K/K', A[\l^\infty](K)) \ .\]
One simply observes that $\Gal(K/K')$ is still an $\l$-adic Lie group and then applies Lemma \ref{CohoBound}.
Hence $Ker(b)$ and $Coker(b)$ are cofinitely generated $\Z_\l$-modules and are finite if
$A[\l^\infty](K)$ is finite.

\subsection{The map $c$} For every prime $v\in\M_F$ let $w'$ be a place of $K'$ lying above $v$. As in Section
\ref{Mapclnotp}, from the Inf-Res sequence, one gets
\[ Ker(d_w) = H^1(K_w/K'_{w'},A[\l^\infty](K_w)) \ .\]
Since $\Gal(K_w/K'_{w'})$ is an $\l$-adic Lie group, every $Ker(d_w)$ is a cofinitely generated $\Z_\l$-module.
The result for $Ker(c)$ will follow from $Ker(d_w)=0$ for all but finitely many primes.
As before the Kummer sequence provides the following diagram
\[ \xymatrix{\ar @{} [dr]
 H^1(K'_{w'},A[\l^{\infty}]) \ar[d]^{d_w} \ar@{^{(}->}[r] &
H^1(K'_{w'},A(F_{v}^s)) \ar[d]^{f_w} \\
H^1(K_w,A[\l^{\infty}]) \ar@{^{(}->}[r] & H^1(K_w,A(F_v^s)) } \]
and, from it, one has the inclusion
\[ Ker(d_w)\iri Ker(f_w)\simeq H^1( K_w/K'_{w'}, A(K_w)) \]
or, more precisely, the isomorphism
\[ Ker(d_w)\simeq H^1( K_w/K'_{w'}, A(K_w))[\l^\infty]\ . \]
Hence we are not going to consider places which are totally split in $K/K'$, because they obviously provide
$Ker(d_w)=0$.

\subsubsection{Unramified primes}
Let $v$ be unramified in $K/F$, then, since $\Gal(K'_{w'}/F_v)\simeq \Z_\l$, $K'_{w'}$ is the maximal
unramified pro-$\l$-extension of $F_v$ and the $\l$-part of
$\Gal(K_w/K'_{w'})$ is trivial. Therefore the $\l$-part of the (torsion) module
$H^1( K_w/K'_{w'}, A(K_w))$ is trivial as well.\\

\noindent Because of the splitting of primes in $K'\,$, we are already left
with finitely many places. So $Ker(c)$ is a cofinitely generated $\Z_\l$-module and the first statement
on $Ker(a)$ and $Coker(a)$ being cofinitely generated over $\Z_\l$ is proved.

\subsubsection{Ramified primes}
Let $v\in S$ and recall that $K'_{w'}$ is not a local field anymore.
If for any $w|w'|v$ the group $A[\l^\infty](K_w)$ is finite, then, by Lemma \ref{CohoBound},
$H^1(K_w/K'_{w'},A[\l^\infty](K_w))$ is finite.
\end{proof}

\begin{cor}\label{FinGenH}
In the setting of Theorem \ref{CTModlneqp}, one has
\begin{itemize}
\item[{\bf 1.}] if $Sel_A(K')_\l$ is a cofinitely generated $\Z_\l$-module, then $Sel_A(K)_\l^\vee$
is finitely generated over $\L(H)$;
\item[{\bf 2.}] suppose that there exists an open, soluble, uniform and pro-$\l$
subgroup $H'$ of $H$; if the groups $Sel_A(K^{H'})_\l$ and $A[\l^\infty](K)$ are finite and for any $w|w'|v\in S$ also $A[\l^\infty](K_w)$
is finite, then $Sel_A(K)_\l^\vee$ is a torsion $\L(H)$-module.
\end{itemize}
\end{cor}

\begin{proof}
The arguments are the same we used for the analogous results for $\L(G)$-modules.
\end{proof}

\begin{rem}
{\em To get $Sel_A(K)_\l^\vee$ finitely generated over $\L(G)$ we only need
to assume that finitely many primes of $F$ are ramified in the Lie extension $K/F$.
Moreover if $Sel_A(\F_p^{(\l)}F)_\l$ is a cofinitely generated $\Z_\l$-module, then $Sel_A(K)_\l^\vee$
is finitely generated also over $\L(H)$. Since $H$ has infinite index in $G$, being finitely
generated over $\L(G)$ and $\L(H)$ implies that $Sel_A(K)_\l^\vee$ is $\L(G)$-torsion
(assuming that $G$ does not contain any element of order $\l$). So
this is another way to obtain torsion $\L(G)$-modules without assuming the finiteness of
$A[\l^\infty](K)$ (which we needed in Corollary \ref{CTCorTor} and is obviously false for
$K=F(A[\l^\infty])\,$). Moreover, as mentioned in the introduction, proving the finitely
generated condition over $\L(G)$ and $\L(H)$ (as done also for all the cases included in Corollary
\ref{FinGenH}) yields that $Sel_A(K)_\l^\vee$ is in the category $\mathfrak{M}_H(G)\,$, i.e.,
allows us to define a characteristic element for $Sel_A(K)_\l^\vee$ (without the $\L(H)$-module
structure, proving that a module is $\L(G)$-torsion is not enough in the non-commutative case). }
\end{rem}

\begin{exe}\label{ExTorField}
{\em Take $K=F(A[\l^\infty])$ as in Section \ref{TorExt}. This kind of extension realizes
na\-tu\-rally most of our assumptions. First of all, from \cite{ST}, $S$ is just the set of places of bad
reduction for $A$ and it is obviously finite (moreover $S=S^*$ in the notations of Section \ref{Mapclnotp}).
As a consequence, by Theorem \ref{fingenlneqp} we get
$Sel_A(K)_\l^\vee$ always finitely generated over $\L(G)$. Then, since $\Gal(K/F)$ embeds
in $GL_{2g}(\Z_\l)$, it is easy to see that for $\l > 2g+1$ the Galois group contains no elements of
order $\l$ so it makes sense to look for torsion modules.
By the Weil-pairing we can take $H$ such that $K'/F$ is the unramified $\Z_\l$-extension
of the constant field of $F$. Because of our choice of $H$, primes in $K'$ above those in $S$ are
finitely many. So, thanks to Theorem \ref{CTModlneqp} if $Sel_A(K')_\l^\vee$ is finitely generated
over $\Z_\l$, then $Sel_A(K)_\l^\vee$ is also finitely generated over $\L(H)$ (hence $\L(G)$-torsion).\\
One can provide examples of Selmer groups $Sel_A(K')_\l$ cofinitely generated over $\Z_\l$ in the work
of Pacheco (\cite[Proposition 3.6]{P}, which generalizes to abelian varieties the analogous
statement of Ellenberg in \cite[Proposition 2.5]{E} for elliptic curves). For more details on
the application of Ellenberg's results to the non-commutative setting of $F(A[\l^\infty])/F$
(like computations of coranks and Euler characteristic of Selmer groups) see Sechi's (unpublished)
PhD thesis \cite{S}. }
\end{exe}

\section {Control theorem for $\l=p$}\label{SecCTl=p}
In order to work with the $p$-torsion part we need to use flat cohomology as mentioned in
Section \ref{SecNot}. We will work with diagram \eqref{FondDiag} but, as we will see, the
kernels and cokernels appearing in the snake lemma sequence will still be described in terms
of Galois cohomology groups. The main difference is provided by the fact that the images of
the local Kummer maps will be nontrivial.\\
To handle the local kernels we shall need the following

\begin{lem}\label{RamPrl=p}
Let $E$ be a local function field and let $L/E$ be an (infinite) $p$-adic Lie extension. Let $A$
be an abelian variety defined over $E$ with good ordinary reduction. Let $I$ be the inertia group
in $\Gal(L/E)$ and assume it is nontrivial. Then $H^1(L/E,A(L))$ (Galois cohomology group) is a
cofinitely generated $\Z_p$-module. Moreover, if $I$ has finite index (i.e., it is open) in
$\Gal(L/E)$, then $H^1(L/E,A(L))$ is finite.
\end{lem}

\begin{proof}
Let $\wh{A}$ (resp. $\ov{A}$) be the formal group associated to $A$ (resp. the reduction of $A$
at the prime of $E$). Because of good reduction the natural map $A(E')\rightarrow \ov{A}(\F_{E'})$
is surjective for any extension $E'/E$. Hence the sequence
\[ \wh{A}(\mathcal{O}_L) \iri A(L) \sri \ov{A}(\F_L) \]
is exact. Taking $\Gal(L/E)$ cohomology (and recalling that $A(E)\rightarrow \ov{A}(\F_E)$ is
surjective), one gets
\[ H^1(L/E,\wh{A}(\mathcal{O}_L)) \iri H^1(L/E,A(L)) \rightarrow H^1(L/E,\ov{A}(\F_L)) \]
and we will focus on the right and left terms of this sequence from now on.

\noindent By \cite[Theorem 2 (a)]{T}, one has an isomorphism
\[ H^1(E,\wh{A}(\mathcal{O}_{\ov{E}})) \simeq Hom(\ov{B}[p^\infty](\F_E),\Q_p/\Z_p) \]
(the statement of the theorem requires a $\Z_p^d$-extension but part (a) holds independently of that).
Therefore the inflation map provides an inclusion
\[ H^1(L/E,\wh{A}(\mathcal{O}_L)) \iri Hom(\ov{B}[p^\infty](\F_E),\Q_p/\Z_p) \]
and, since $\ov{A}$ and $\ov{B}$ are isogenous, the term on the right is finite of order
$|\ov{B}[p^\infty](\F_E)|=|\ov{A}[p^\infty](\F_E)|$.\\

\noindent For $H^1(L/E,\ov{A}(\F_L))$ we consider two cases:

\subsection{Case 1: $I$ is open in $\Gal(L/E)$.}
Just considering $p$-parts we can assume that $I$ is open in a $p$-Sylow $P_p$ of $\Gal(L/E)$
and we let $L^I$ be its fixed field. Let $\ov{P_p}:=\Gal(L^I/E)$ (a finite group) and note that,
since $L/L^I$ is a totally ramified extension, one has that
$\F_L=\F_{L^I}$ is still a finite field. The Inf-Res sequence reads as
\[ H^1(L^I/E,\ov{A}(\F_{L^I})) \iri H^1(L/E,\ov{A}(\F_L)) \rightarrow H^1(L/L^I,\ov{A}(\F_L)) \ .\]
The group on the left is obviously finite and for the one on the right we can use
\cite[Proposition 3.1]{G} as done previously in Lemma \ref{CohoBound}
(because $I$ is still a $p$-adic Lie group) to get
\[ |H^1(L/L^I,\ov{A}(\F_L))| \leqslant |\ov{A}(\F_L))|^{\tilde{d}_1} \]
where the exponent $\tilde{d}_1=\tilde{d}_1(\Gal(L/E))$ only depends on $\Gal(L/E)$.

\subsection{Case 2: $I$ has infinite index in $\Gal(L/E)$.}
We use the same sequence
\[ H^1(L^I/E,\ov{A}(\F_{L^I})) \iri H^1(L/E,\ov{A}(\F_L)) \rightarrow H^1(L/L^I,\ov{A}(\F_L)) \]
but now $\F_L$ is not a finite field anymore: indeed it contains the $\Z_p$-extension of $\F_E$
(because unramified extensions come from extensions of the field of constants).
For the group on the right we again use Lemma \ref{CohoBound}
to prove that $H^1(L/L^I,\ov{A}(\F_L))$ is a cofinitely generated $\Z_p$-module.
The only difference with case {\bf 1} is that now the $p$-divisible part of $\ov{A}(\F_L)$
might come into play (moreover note that $H^1(L/L^I,\ov{A}(\F_L))$ is a torsion abelian group,
hence its $p$-primary part is exactly $H^1(L/L^I,\ov{A}(\F_L)[p^\infty])$ and $A(\F_L)[p^\infty]$
has finite $\Z_p$-corank). For $H^1(L^I/E,\ov{A}(\F_{L^I}))$ we observe that the $p$-part of
$\Gal(L^I/E)$ is isomorphic to $\Z_p$ and that the subgroup of $\ov{A}(\F_{L^I})$ fixed by
that $p$-part is finite. Hence, by \cite[Lemma 3.4 and Remark 3.5]{BL2}, one has that
\[ |H^1(L^I/E,\ov{A}(\F_{L^I}))| \leqslant  |\ov{A}(\F_{L^I})/(\ov{A}(\F_{L^I}))_{div}| \]
is finite.
\end{proof}

\noindent The lemma provides a bound for the order of $H^1(L/E,A(L))$ which (when
$I$ is open in $\Gal(L/E)\,$) can be written in terms of $\ov{A}[p^\infty](\F_E)$ and
$\ov{A}(\F_L)$. If the inertia is infinite (i.e., if $I$ has order divisible by arbitrary high
powers of $p$ or, as we will say from now on, {\em has order divisible by} $p^\infty$) we
can prove that the inflation map
\[ H^1(L/E,\wh{A}(\mathcal{O}_L)) \iri H^1(E,\wh{A}(\mathcal{O}_{\ov{E}})) \]
is actually an isomorphism but, since this is not going to improve the bound, we decided to keep
this statement out of the lemma and we include it here only for completeness (the proof is just a
generalization of the one provided for \cite[Theorem 2 (b)]{T}).

\begin{prop}
In the same setting of the previous lemma, assume that $I$ has order divisible by $p^\infty\,$, then
\[ H^1(L/E,\wh{A}(\mathcal{O}_L)) = H^1(E,\wh{A}(\mathcal{O}_{\ov{E}})) \]
\end{prop}

\begin{proof} The inflation map immediately provides one inclusion so we need to prove the reverse one.
By \cite[Corollary 2.3.3]{T}, $H^1(L/E,A(L))$ is the annihilator of $N_{L/E}(B(L))$ with respect to
the local Tate pairing (where $N_{L/E}$ is the natural norm map). This provides an isomorphism
\begin{equation}\label{IsoTatePair} H^1(L/E,A(L)) \simeq (B(E)/N_{L/E}(B(L)))^\vee \ . \end{equation}
Working as in \cite[Section 2]{T} (in particular subsections 2.3
and 2.6 in which the results apply to general Galois extensions), one verifies that
$H^1(E,\wh{A}(\mathcal{O}_{\ov{E}}))$ is the annihilator of $\wh{B}(\mathcal{O}_E)$ with respect to the local
pairing. Moreover, since $\wh{A}(\mathcal{O}_{\ov{E}})$ is the kernel of the reduction map, one has
\[ \begin{array}{ll} H^1(E,\wh{A}(\mathcal{O}_{\ov{E}})) \cap H^1(L/E,A(L)) & \subseteq
Ker \Big\{ H^1(L/E,A(L)) \rightarrow  H^1(L/E,\ov{A}(\F_L)) \Big\} \\
\ & = H^1(L/E,\wh{A}(\mathcal{O}_L))\ . \end{array} \]
Therefore it suffices to show $H^1(E,\wh{A}(\mathcal{O}_{\ov{E}})) \subseteq H^1(L/E,A(L))$ and,
because of the isomorphism \eqref{IsoTatePair}, this is equivalent to proving
$N_{L/E}(B(L)) \subseteq \wh{B}(\mathcal{O}_E)$.\\
Take $\alpha \in B(L)$ and put $x=N_{L/E}(\alpha)$: we can assume that $\alpha$ belongs to the $p$-part
of $B(L)$, so that $x$ is in the $p$-part of $B(E)$. Let $E'$ be an intermediate field containing
$L^I$ and such that $p^\infty$ divides $[E':L^I]$. Consider $z=N_{L/E'}(\alpha)\in B(E')$ and let
$\ov{z}$ be the image of $z$ in $\ov{B}(\mathbb{F}_{E'})=\ov{B}(\mathbb{F}_{L^I})$
(which is a torsion group). Hence $\ov{z}$ has finite order, say $p^m$, and we can find a
field $E''$ between $E'$ and $L^I$ such that $p^m|[E'':L^I]\,$. Put
$y=N_{L/E''}(\alpha)=N_{E'/E''}(z)\in B(E'')$, so that $x=N_{E''/E}(y)$
and note that $\ov{y}$ (the image of $y$ in $\ov{B}(\mathbb{F}_{E''})=\ov{B}(\mathbb{F}_{L^I})\,$)
has order dividing $p^m\,$. Since $\Gal(E''/L^I)$ fixes $\ov{y}$, one has that $N_{E''/L^I}(\ov{y})$
is trivial in $\ov{B}(\mathbb{F}_{L^I})$, i.e., $N_{E''/L^I}(y)\in \wh{B}(\mathcal{O}_{L^I})$
(which is the kernel of the reduction map). Hence
\[ x = N_{E''/E}(y) =  N_{L^I/E}(N_{E''/L^I}(y)) \in N_{L^I/E}(\wh{B}(\mathcal{O}_{L^I}))
\subseteq \wh{B}(\mathcal{O}_E)\ .\]
\end{proof}

\noindent Now we proceed with our control theorem.

\begin{thm}\label{mainthm2}
Assume that all ramified primes are of good ordinary or split multiplicative reduction.
Then, for any finite extension $F'/F$ contained in $K$,
the kernels and cokernels of the map
\[ a_{K/F'}: Sel_A(F')_p \to Sel_A(K)^{\Gal(K/F')}_p \]
are cofinitely generated $\Z_p$-modules. If all primes in $S$ and all primes of bad reduction
have decomposition groups open in $G$, then the coranks of kernels and cokernels
are bounded independently of $F'\,$.
Moreover if the group $A[p^\infty](K)$ is finite, all places in
$S$ are of good reduction and have inertia groups open in their decomposition groups,
then the kernels and cokernels are finite (of bounded order if the primes of bad reduction
have open decomposition group).
\end{thm}

\begin{proof}
We work with the usual diagram \eqref{FondDiag}, where now we cannot substitute flat cohomology with
Galois cohomology and in the groups $\G_A(\cdot)$ the images of the local Kummer maps are nontrivial
(in general).\\
Since the map $\X_K \to \X_{F'}$ is a Galois covering with Galois group $\Gal(K/F')$, the
Hochschild-Serre spectral sequence applies and we will study: \begin{itemize}
\item[-] $Ker(b_{K/F'})\simeq H^1(K/F',A[p^{\infty}](K))$;
\item[-] $Coker(b_{K/F'})\subseteq H^2(K/F',A[p^{\infty}](K))$;
\item[-] $Ker(c_{K/F'})$
\end{itemize}
noting that the cohomology groups on the right are Galois cohomology groups
(see \cite[III.2.21 (a), (b) and III.1.17 (d)]{Mi2}).

\subsection{The map $b_{K/F'}$.} Just use Lemma \ref{CohoBound}.

\subsection{The maps $c_{K/F'}$.} As before we simply work with the maps
\[d_w: H^1_{fl}(X_{F'_{v'}},A[p^{\infty}])/Im\,\k_{v'}\to H^1_{fl}(X_{K_w},A[p^{\infty}])/Im\,\k_w\ . \]
From the Kummer sequence one gets a diagram
\[ \xymatrix{\ar @{} [dr]
 H^1_{fl}(X_{F'_{v'}},A[p^{\infty}])/Im\,\k_{v'} \ar[d]^{d_w} \ar@{^{(}->}[r] &
H^1_{fl}(X_{F'_{v'}},A)[p^\infty] \ar[d]^{f_w} \\
H^1_{fl}(X_{K_w},A[p^{\infty}])/Im\,\k_w \ar@{^{(}->}[r] & H^1_{fl}(X_{K_w},A)[p^\infty] \ ,} \]
and (from the Inf-Res sequence)
\[ Ker(d_w)\iri Ker(f_w)\simeq H^1( K_w/F'_{v'}, A(K_w))[p^\infty]\ . \]
Before moving on observe that $K_w/F'_{v'}$ is a $p$-adic Lie extension because the
decomposition group of any place is closed in $\Gal(K/F')$. Now we distinguish two cases
depending on their behaviour in $K/F'$ (and, as usual, we do not consider primes which split
completely because they give no contribution to $Ker(c_{K/F'})$).

\subsubsection{Unramified primes}
If $v'$ is unramified, from \cite[Proposition I.3.8]{Mi1}, we have
\[ H^1( (F'_{v'})^{unr}/F'_{v'}, A((F'_{v'})^{unr}))=
H^1( (F'_{v'})^{unr}/F'_{v'}, \pi_0(\mathcal{A}(F'_{v'})_0)) \]
where $\pi_0(\mathcal{A}(F'_{v'})_0)$ is the set of connected components of the closed fiber of the
N\'eron model $\mathcal{A}(F'_{v'})$ of $A$ at $v'\,$. The latter is a finite module and it is trivial
when $v'$ is a place of good reduction. Because of the inflation map
\[ H^1( K_w/F'_{v'}, A(K_w)) \iri H^1( (F'_{v'})^{unr}/F'_{v'}, A((F'_{v'})^{unr})) \]
the same holds for $Ker(d_w)$ as well. Note also that in finite unramified (hence cyclic) extensions,
since the group $\pi_0(\mathcal{A}(F'_{v'})_0)$ is finite, the order of the $H^1$ is equal to
the order of the $H^0$ which is uniformly bounded (see for example \cite[Lemma 3.3.1]{T}). Thus the order of
$H^1( (F'_{v'})^{unr}/F'_{v'}, \pi_0(\mathcal{A}(F'_{v'})_0))$ is bounded independently
of $F'\,$.

\subsubsection{Ramified primes}\label{RamSplitRed}
If $v'$ is of good ordinary reduction, just observe that $\Gal(K_w/F'_{v'})$ is a
$p$-adic Lie group with nontrivial inertia and apply Lemma \ref{RamPrl=p}.\\
We are left with ramified primes of split multiplicative reduction. Let $v'$ be such a prime.
We have the exact sequence coming from Mumford's uniformization
\[ < q'_{A,1},\dots,q'_{A,g} > \iri ((\overline{F_v})^*)^g \sri A(\overline{F_v}) \]
(where we recall that $g$ is the dimension of the variety $A$, the $q'_{A,i}$ are parameters
in $(F'_{v'})^*$ and the morphisms behave well with respect to the Galois action).\\
Taking cohomology (and using Hilbert's Theorem 90) one finds an injection
\[ H^1( K_w/F'_{v'},A(K_w)) \iri H^2(K_w/F'_{v'},< q'_{A,1},\dots,q'_{A,g} >) \ .\]
Since $q'_{A,i}\in (F'_{v'})^*$, the action of the Galois group is trivial on them and
we have an isomorphism of Galois modules $< q'_{A,1},\dots,q'_{A,g} > \simeq \Z^g\,$. Hence
\[ H^2(K_w/F'_{v'},< q'_{A,1},\dots,q'_{A,g} >) \simeq H^2(K_w/F'_{v'},\Z^g)
\simeq ((\Gal(K_w/F'_{v'})^{ab})^\vee)^g \]
(where the last one is the Pontrjagin dual of the maximal abelian quotient of $\Gal(K_w/F'_{v'})\,$).
The last one is a cofinitely generated $\Z_p$-module, since $\Gal(K_w/F'_{v'})^{ab}\,$ is virtually a finitely
generated $\Z_p$-module. Indeed, $\ov{[\Gal(K_w/F'_{v'}),\Gal(K_w/F'_{v'})]}$ (the topological
closure of the commutators) is a closed normal subgroup of $\Gal(K_w/F'_{v'})$ and their quotient
is still a $p$-adic Lie group (see \cite[Theorem 9.6 (ii)]{DdSMS} and, for the $\Z_p$-module
structure, \cite[Theorem 4.9]{DdSMS} and \cite[Theorem 4.17]{DdSMS}).
\end{proof}

\begin{rem}{\em \ \begin{itemize}
\item[{\bf 1.}] {The hypothesis on the reduction of A at ramified primes are necessary. Indeed, if $v$ is a
ramified prime of good supersingular reduction and $\Gal(K/F)\simeq \Z_p^d$ (a deeply
ramified extension in the sense of \cite[Section 2]{CG}), then K.-S. Tan has shown that
$ H^1(K_w/F'_{v'},A(K_w))[p^\infty]$ has infinite $\Z_p$-corank (see \cite[Theorem 3.6.1]{T2}).}
\item[{\bf 2.}] {When $A=E$ is an elliptic curve it is easy to see that the number of torsion
points of $p$-power order in a separable extension of $F$ is finite (see, for example,
\cite[Lemma 4.3]{BL}). For a general abelian variety $A$, in \cite[Lemma 2.5.1]{T} K.-S. Tan shows
that $A[p^\infty](K)=A[p^\infty](K\cap F^{unr})$. Hence, for a $p$-adic Lie extension,
we are interested only in the number of $p$-power torsion points in $\F_p^{(p)}F$ (or
$\F_p^{(p)}F'$ for some finite unramified extension $F'$ of $F$).}
\end{itemize}}
\end{rem}

\section{$\L$-modules for $\l=p$}\label{SecModStrl=p}
 The Selmer groups are again $\L(G)$-modules and we will investigate their structure as in the case $l\neq p$.
We assume that $G$ does not contain any element of order $p$.

\subsection{Structure of $Sel_A(K)_p^\vee$ as $\L(G)$-module}

\begin{thm}\label{fingenl=p} With the above notations, if all places in $S$ are of
good ordinary or split multiplicative reduction, then $Sel_A(K)_p^\vee$ is a
finitely generated $\L(G)$-module.
\end{thm}

\begin{proof} This is the same proof of Theorem \ref{fingenlneqp}, using Theorems \ref{mainthm2}
and \ref{BHNak}.
\end{proof}

\begin{thm}\label{Torsl=p}
Assume that there exists an open soluble, uniform and pro-$p$ subgroup $G'$ of $G$.
Suppose that $A[p^\infty](K)$ and $Sel_A(F')_p^\vee$, where $F'$ is the fixed field of $G'$, are finite.
If all ramified places of $F$ are of good ordinary reduction for $A$ and have inertia groups open in their
decomposition groups, then $Sel_A(K)_p^\vee$ is a torsion $\L(G)$-module.
\end{thm}

\begin{proof}It is sufficient to show that $Sel_A(K)_p^\vee$ is a torsion $\L(G')$-module
(classical meaning). In order to do this just use Theorems \ref{mainthm2} and \ref{BHNak}.
\end{proof}

\subsection{Structure of $Sel_A(K)_p^\vee$ as $\L(H)$-module}\label{Hmodl=p} Assume that there
exists a closed normal subgroup $H$ in $G$ such that $G/H=\Gamma\simeq \Z_p$ and let $K'$ be
its fixed field.\\
We need the following

\begin{lem}\label{LocalOrNot}
With notations as above, let $K':=K^H$ be the fixed field of $H$. If $v\in \M_F$ is unramified in $K/F$, then
\begin{itemize}
\item[{\bf 1.}] $v$ splits completely in $K'/F$ or
\item[{\bf 2.}] the decomposition group, in $H=\Gal(K/K')$, of any prime $w'$ of $K'$ dividing $v$
is finite and has order prime to $p$.
\end{itemize}
\end{lem}

\begin{proof}
\noindent Consider the diagram
\[ \begin{xy}
(0,15)*+{K}="v1";(15,18)*+{K_w}="v2";
(0,0)*+{K'}="v3";(15,3)*+{K'_{w'}}="v4";
(0,-15)*{F}="v5";(15,-12)*+{F_v}="v6";
{\ar@{-}^{H} "v1";"v3"};{\ar@{-} "v1";"v2"};
{\ar@{-}^{\Gamma} "v3";"v5"};{\ar@{-} "v3";"v4"};
{\ar@{-} "v5";"v6"};{\ar@{-} "v2";"v4"};
{\ar@{-} "v4";"v6"};{\ar@{-}@/_{1.5pc}/_{G} "v1";"v5"}
\end{xy} \]
where $w|w'|v$.\\
Assume $v$ is unramified, then, since $\Gal(K_w/F_v)$ is still a $p$-adic Lie group and the $p$-part of
an unramified $p$-adic Lie extension of local fields is at most a $\Z_p$-extension,
we have to deal with two cases.
\begin{itemize}
\item[{\bf a.}] {\em $\Gal(K_w/F_v)$ is finite.} Then $\Gal(K'_{w'}/F_v)$ is a finite subgroup of
$\Gamma\simeq \Z_p\,$, hence it is trivial and $v$ splits completely if $K'/F$ (i.e., {\bf 1} holds).
\item[{\bf b.}] {\em $\Gal(K_w/F_v)\simeq \Gamma$.} Then $\Gal(K'_{w'}/F_v)$ can be trivial or isomorphic to $\Gamma$.
If it is trivial we are back to case {\bf 1}. If it is $\Gamma$, then $K'_{w'}$ is the unramified $\Z_p$-extension
of $F_v\,$. Hence the extension $K_w/K'_{w'}$ is unramified of (finite) order prime to $p$ (i.e., {\bf 2} holds).
\end{itemize}
\end{proof}

\noindent As in the $\l\neq p$ case, we will show a slightly modified version of Theorem \ref{mainthm2}.

\begin{thm}\label{CTModl=p} Assume that
\begin{itemize}
\item[{\bf 1.}] {all places in $S$ are of split multiplicative reduction for $A$;}
\item[{\bf 2.}] {all places in $S$ and all places of bad reduction for $A$ split in finitely many primes in $K'/F$.}
\end{itemize}
Then the map
\[ a: Sel_A(K')_p\to Sel_A(K)^H_p\]
has cofinitely generated kernel and cokernel (viewed as $\Z_p$-modules).
\end{thm}

\begin{proof}We go directly to the local kernels for places which do not split completely
\[ Ker(d_w) \simeq H^1( K_w/K'_{w'}, A(K_w))[p^\infty]\ .\]

\subsubsection{Unramified primes of good reduction} Let $v$ be an unramified prime of good reduction for $A$. From
Lemma \ref{LocalOrNot} we know that $K'_{w'}=F_v$ is a local field or $\Gal(K_w/K'_{w'})$ has finite order
prime with $p$. In the first case \cite[Ch. I, Proposition 3.8]{Mi1} shows that $H^1(K_w/K'_{w'},A(K_w))=0$
and we get our claim. In the second case the $p$-part of the torsion module $H^1(K_w/K'_{w'},A(K_w))$ is
obviously trivial.\\

\noindent Because of our hypothesis on the splitting of primes we are already left with finitely many
local kernels, now we check the behaviour of the remaining ones.

\subsubsection{Unramified primes of bad reduction} As above if $\Gal(K'_{w'}/F_v)\simeq\Gamma$ (i.e., $K'_{w'}$
is not a local field), then the $p$-part of $H^1(K_w/K'_{w'},A(K_w))$ is trivial. The other case (i.e.,
$K'_{w'}=F_v\,$) cannot happen because we are assuming that bad reduction primes do not split completely
in $K'\,$.

\subsubsection{Ramified primes} For the ramified ones we use Mumford's parametrization as in Section \ref{RamSplitRed}.
First we get an exact sequence
\[ < q_{A,1},\dots,q_{A,g}> \iri (K_w^*)^g \sri A(K_w) \]
(where we can always assume that the periods are in the base field $F_v\,$).
Then $\Gal(K_w/K'_{w'})$-cohomology provides an injection
\[ H^1(K_w/K'_{w'}, A(K_w)) \iri H^2(K_w/K'_{w'},\Z^g) \simeq ((\Gal(K_w/K'_{w'})^{ab})^\vee)^g\ . \]
\end{proof}

\noindent As a consequence we have the following (which generalizes \cite[Theorem 1.9]{OT}).

\begin{cor}\label{ModStrl=p}
In the setting of Theorem \ref{CTModl=p}, assume that $Sel_A(K')_p^\vee$ is a finitely generated
$\Z_p$-module. Then $Sel_A(K)_p^\vee$ is finitely generated over $\L(H)$ (hence torsion over $\L(G)\,$).
\end{cor}

\subsection{Final summary}
Take the function fields $K$ and $K'$ such that the hypothesis on the splitting of primes in Theorem
\ref{mainthm2} is verified.\\
For $Sel_A(K)^\vee_p$ to be finitely generated as $\L(G)$-module we need just to assume that all
primes in $S$ are of good ordinary or split multiplicative reduction for $A$. To move a step further
and find $\L(G)$-torsion modules we can assume
\begin{itemize}
\item[{\bf 1.}] all ramified primes are of good ordinary reduction and have open inertia groups;
\item[{\bf 2.}] $A[p^\infty]$ and $Sel_A(F')^\vee_p$ are finite (where $F'$ is the fixed field of
an open soluble, uniform and pro-$p$ subgroup $G'$ of $G$, if such a subgroup exists),
\end{itemize}
and use Theorem \ref{Torsl=p}.\\
Another way to find $\L(G)$-torsion modules is provided by Corollary \ref{ModStrl=p}. It somehow
complements the previous one because it requires a different type of reduction
for the ramified places. The assumptions for this case are
\begin{itemize}
\item[{\bf 3.}] all ramified primes are of split multiplicative reduction;
\item[{\bf 4.}] $Sel_A(K')_p$ is cofinitely generated over $\Z_p\,$.
\end{itemize}
One example for this second case is given by the usual arithmetic extension $K' = \F_p^{(p)}F$
(for which the hypothesis {\bf 1} above obviously does not hold).

\begin{rem}
{\em We found conditions to get $\L(G)$-torsion modules for different types of reduction
for the ramified primes, but we remind that, to get a characteristic element for a
$\L(G)$-torsion module for a non-commutative group $G$, one needs to examine the $\L(H)$-module
structure as well. Hence our results provide characteristic elements for $Sel_A(K)_p^\vee$
only when the ramified primes are of split multiplicative reduction.}
\end{rem}

\noindent Andrea Bandini\\
Universit\`a degli Studi di Parma - Dipartimento di Matematica\\
Parco Area delle Scienze, 53/A - 43124 Parma - Italy\\
e-mail: andrea.bandini@unipr.it\\

\noindent Maria Valentino\\
Universit\`a della Calabria - Dipartimento di Matematica\\
via P. Bucci - Cubo 31B - 87036 Arcavacata di Rende (CS) - Italy\\
e-mail: valentino@mat.unical.it

\end{document}